\documentclass[preprint,10pt]{elsarticle}
\usepackage[letterpaper,top=2.75cm,bottom=2.75cm,left=2.75cm,right=2.75cm,marginparwidth=2.75cm]{geometry}

\usepackage[T1]{fontenc}
\usepackage[utf8]{inputenc} 
\usepackage{microtype}

\usepackage[letterpaper,top=2.75cm,bottom=2.75cm,left=2.75cm,right=2.75cm,marginparwidth=2.75cm]{geometry}

\usepackage{mathtools}  
\usepackage{amsthm,amsfonts,bm}
\newtheorem{assumption}{Assumption}

\newtheorem{proposition}{Proposition}
\DeclareMathOperator*{\argmax}{arg\,max}
\DeclareMathOperator*{\argmin}{arg\,min}

\usepackage{graphicx,booktabs,multirow,dcolumn,subcaption}
\usepackage{algorithm2e}
\usepackage{listings}

\biboptions{numbers,sort&compress}

\usepackage[backref=page]{hyperref}
\usepackage[nameinlink,capitalize]{cleveref}
\Crefname{equation}{Eq.}{Eqs.}
\crefname{equation}{Eq.}{Eqs.}

\usepackage{xcolor}

\Crefname{equation}{Eq.}{Eqs.}

\journal{Elsevier}

\begin{document}

\begin{frontmatter}

\title{Bi-fidelity Interpolative Decomposition for Multimodal Data}

\author[1]{Murray Cutforth}
\ead{mcc4@stanford.edu}
\author[2]{Tiffany Fan}
\ead{tiffan@stanford.edu}
\author[3]{Tony Zahtila}
\ead{tzahtila@stanford.edu}
\author[4]{Alireza Doostan\corref{mycorrespondingauthor}}
\cortext[mycorrespondingauthor]{Corresponding author}
\ead{doostan@colorado.edu}
\author[1]{Eric Darve}
\ead{darve@stanford.edu}
\affiliation[1]{address={Mechanical Engineering, Stanford University}}
\affiliation[2]{address={Computational and Mathematical Engineering, Stanford University}}
\affiliation[3]{address={Center for Turbulence Research, Stanford University}}
\affiliation[4]{address={Ann and H.J. Smead Department of Aerospace Engineering Sciences, University of Colorado, Boulder}}

\begin{abstract}

Multi-fidelity simulation is a widely used strategy to reduce the computational cost of many-query numerical simulation tasks such as uncertainty quantification, design space exploration, and design optimization. The reduced basis approach based on bi-fidelity interpolative decomposition is one such approach, which identifies a reduced basis, along with an interpolation rule in that basis, from low-fidelity samples to approximate the corresponding high-fidelity samples. However, as illustrated in the present study, when the model response is multi-modal and mode occupancy is stochastic, the assumptions underpinning this approach may not hold, thus leading to inaccurate estimates. We introduce the multi-modal interpolative decomposition method using bi-fidelity data, an extension tailored for this use case. Our work is motivated by a complex engineering application: a laser-ignited methane-oxygen rocket combustor evaluated over uncertain input parameters, exhibiting a bifurcation-like phenomenon in some regions of parameter space. Unlike the standard bi-fidelity interpolative decomposition approach, the proposed method can approximate a dataset of high-fidelity simulations for 16\% of the cost, while maintaining relatively high correlation (0.70--0.90) with parameter sensitivities.

\end{abstract}


\end{frontmatter}

\section{Introduction}
\label{sec:introduction}

Computational cost is a ubiquitous limiting factor in the simulation of realistic physical and engineering systems, particularly when evaluated over a large input parameter space as is required in many-query tasks such as uncertainty quantification (UQ) \cite{doostan2009least,doostan2011non,zhang2021modern}, design optimization \cite{queipo2005surrogate}, and inverse problems \cite{hasanouglu2021introduction}. The strategy of multi-fidelity modeling \cite{peherstorfer2018survey, ng2012multifidelity, giles2015multilevel} has emerged to address this, where two or more different fidelity levels are combined, exploiting correlation between levels. Lower-fidelity in this context means a cheaper but less accurate model, such as using a data-driven surrogate model, reducing resolution, or using simplified governing equations. 

Among multi-fidelity UQ techniques, this work utilizes the bi-fidelity interpolative decomposition  (bi-fidelity ID) introduced in \cite{doostan2016bi,hampton2018}, a linear algebra-based technique closely related to \cite{doostan2007stochastic,narayan2014stochastic,newberry2022bi}. This approach identifies a reduced basis, along with an interpolation rule in that basis, from a collection of low-fidelity samples and approximates the high-fidelity output at the same set of parametric inputs using a small number of high-fidelity model evaluations. The bi-fidelity ID procedure is based on the assumption that an identical interpolation rule applies to both fidelity levels.
While successful in many single-mode contexts, \cite{fairbanks2020bi,fairbanks2017low,skinner2019reduced,pinti2022multi}, such an assumption can fail for multi-modal or bifurcating systems, where regions in the input space may exhibit qualitatively distinct solution ``clusters'' (or ``branches'') and cluster membership (branch occurrence) is stochastic. In these cases, naïvely transferring a single interpolation rule from the low-fidelity data may produce large approximation errors in the high-fidelity predictions.

This paper presents an extension of the bi-fidelity ID approach tailored for data arising from multi-modal or bifurcating responses driven by stochasticity. The main idea is to incorporate a mixture-model view of the low- and high-fidelity data, which allows conditional interpolation weights to be employed. By conditioning the high-fidelity interpolation on the actual cluster (or mode) identified in a small subset of high-fidelity data, one can avoid mismatches of the low-fidelity and high-fidelity basis elements that lead to large approximation errors. The method, termed \emph{multi-modal ID}, retains the efficiency of a column-based bi-fidelity approach while extending it to cases where distributions of the quantity of interest (QOI), conditioned on the realization of random inputs, may present multiple solution modes, each requiring a distinct cluster in the low-rank representation. The method description and numerical experiments consider the specific case of a bifurcating solution (two clusters); however, we note that the approach can be straightforwardly extended to more than two clusters.

We organize the paper as follows. After providing a survey of related bi-fidelity and low-rank approximation techniques, we present in Section \ref{sec:method} the existing bi-fidelity interpolative decomposition framework. Section \ref{sec:mmidmethod} then introduces the proposed multi-modal ID approach, including conditions under which it is valid and an algorithmic description. Numerical examples in Section \ref{sec:experiments} illustrate that the multi-modal ID drastically improves approximation quality in multi-modal scenarios.

\section{Related Work}
\label{sec:related-work}

\paragraph{Multi-fidelity surrogate modeling} 

The idea of exploiting multiple model fidelities for efficient uncertainty quantification has been extensively explored in the literature. Early techniques focused on co-kriging \cite{kennedy2000predicting,le2014recursive, le2013multi}. Another popular approach is polynomial chaos expansion, in which polynomial bases are combined across fidelities \cite{ng2012multifidelity,palar2016multi,newberry2022bi}.
More recently, neural network-based representations have become popular, including adaptations for multi-fidelity data such as the neural surrogate model in \cite{meng2020composite,de2020transfer,motamed2020multi,de2022neural}, and the bi-fidelity variational auto-encoder in \cite{cheng2024bi}. While neural network-based models are valuable tools, linear methods such as the proposed method still have their niche due to vastly reduced compute and data requirements.

\paragraph{Interpolative decomposition and low-rank approximation} ID is a low-rank matrix factorization which is distinguished by the fact that one factor consists of a subset of columns from the original data matrix, which is a useful feature in downstream tasks -- the truncated singular value decomposition, which provides the optimal reconstruction \cite{eckart1936approximation} does not have this attribute. There exists a large number of randomized algorithms for the efficient computation of the ID \cite{liberty2007randomized}, and this remains an area of active research, e.g., \cite{dong2023robust}.

In the multi-fidelity setting, the idea of utilizing a low-rank ID was introduced by Doostan \emph{et al.} \cite{doostan2016bi}. The multi-fidelity ID is closely related to the approaches presented in \cite{doostan2007stochastic,narayan2014stochastic,newberry2022bi}. In subsequent work \cite{hampton2018}, an {\it a posteriori} error estimate is derived to assess the utility of a pair of low- and high-fidelity models in producing accurate bi-fidelity data.

\paragraph{Stochastic extensions and mixture models} 

Several authors have investigated ways of blending clustering or classification ideas with low-dimensional subspace approximation. The clustered reduced-order modeling approach in \cite{kaiser2014cluster} decomposes the solution manifold into separate regions or modes, each with its own reduced basis. The classification and cluster-based basis function approach to reduced order modeling in \cite{xiong2022pre} is also similar. However, none of the above approaches leverages multi-fidelity data.

\section{Methodology}
\label{sec:method}

\subsection{Bi-fidelity ID}
\label{sec:bf_id}

We begin by restating the bi-fidelity ID algorithm introduced in \cite{doostan2016bi}. The quantity of interest (QOI) is denoted as $\boldsymbol v(\boldsymbol\xi) \in \mathbb{R}^{M}$, where $\boldsymbol \xi \in \mathbb{R}^d$ is a vector of random inputs with known probability density. We are interested in low-rank approximations to $\boldsymbol v(\boldsymbol\xi)$:
\begin{equation}
\label{eq:id}
    \boldsymbol v(\boldsymbol\xi) \approx \sum_{l=1}^r \boldsymbol v(\boldsymbol\xi_l) c_l(\boldsymbol\xi),
\end{equation}
where it is assumed that a rank $r\ll M$ representation produces an accurate estimate of $\bm{v}(\bm\xi)$. The approach is based on measuring the (discretized) QOI using a low- and a high-fidelity model, denoted by $\boldsymbol v_L(\boldsymbol\xi) \in \mathbb{R}^{M_L}$ and $\boldsymbol v_H(\boldsymbol\xi) \in \mathbb{R}^{M_H}$, respectively. Suppose that the low-rank approximation in (\ref{eq:id}) holds for both $\boldsymbol v_L(\boldsymbol\xi)$ and $\boldsymbol v_H(\boldsymbol\xi)$ with identical coefficients $c_l(\boldsymbol\xi)$. Then, the bi-fidelity approximation to $\boldsymbol v_H(\boldsymbol\xi)$ is obtained by using a low-fidelity dataset to compute the coefficients $c_l(\boldsymbol\xi)$ at $N$ parameter inputs $\{ \boldsymbol \xi_j \}_{j=1}^N$. The high-fidelity model is only evaluated at $r$ parametric inputs corresponding to the input samples $\{ \boldsymbol \xi_l \}_{l=1}^r$, $r \ll N$, identified by low-rank approximation of the low-fidelity data, as explained below. Then, the bi-fidelity approximation is:
\begin{equation}
\label{eq:id2}
    \boldsymbol v_H(\boldsymbol\xi_j) \approx \sum_{l=1}^r \boldsymbol v_H(\boldsymbol\xi_l) c_l(\boldsymbol\xi_j), \quad j=1, 2, ..., N.
\end{equation}
The assumption of identical $c_l(\boldsymbol\xi)$ in the approximation of $\boldsymbol v_L(\boldsymbol\xi)$ and $\boldsymbol v_H(\boldsymbol\xi)$ was discussed in \cite{hampton2018} and conditions on the low- and high-fidelity models were identified for it to hold. In particular, if the two fidelities are related through a linear operator, this approximation is exact given some assumptions discussed in Appendix \ref{app:linearop}.

We now provide further details on the implementation of this method. A summary is provided in Algorithm \ref{alg:deterministicid} of Appendix \ref{app:deterministicid}. 
The $M_L \times N$ low-fidelity data matrix $\bm L$ is the concatenation of the low-fidelity QOI $\boldsymbol v_L(\boldsymbol\xi)$ evaluated at $N$ parametric inputs $\{ \boldsymbol \xi_j \}_{j=1}^N$, i.e., $\bm{L}(:,j) = \bm{v}_L(\bm\xi_j)$. Our goal is to find an approximate factorization $\bm L \approx \bm L_r \bm C_L$ such that $\bm L_r$ is a subset of $r$ columns of $\bm L$, and $\bm C_L$ is the corresponding interpolation (coefficient) matrix. Following \cite{doostan2016bi}, the rank-revealing QR factorization can be applied
\begin{equation}
    \bm L \bm P \approx \bm Q_L \bm R_L,
\end{equation}
where $\bm P$ is an $N \times N$ permutation matrix. By partitioning $\bm R_L$ into a $r \times r$ upper triangular matrix $\bm R_L^{(11)}$ and a $r \times (N-r)$ matrix $\bm R^{(12)}$, we obtain
\begin{equation}
    \bm L \approx \bm Q_L \bm R_L^{(11)} \left[ \bm I \ | \ \left(\bm R_L^{(11)}\right)^{-1} \bm R_L^{(12)} \right]\bm P^T := \bm L_r \bm C_L,
\end{equation}
where $\bm L_r:=\bm Q_L \bm R_L^{(11)}$ is the $M_L \times r$ skeleton matrix consisting of $r$ columns of $\bm L$. 
Here, the $r$ entries in each column of $\bm C_L$ correspond to the $r$ coefficients $c_l(\bm \xi)$ in (\ref{eq:id}) for a specific $\bm \xi$.
In practice, the rank-revealing QR algorithm is iterated with increasing $r$ until $\| \bm L - \bm L_r \bm C_L \|_F$ reaches a pre-specified accuracy, where $\| \cdot\|_F$ denotes the Frobenius norm. 
The $r$ columns of $\bm L$, which are selected in $\bm L_r$ each correspond to a different parametric input to the low-fidelity model. This set of $r$ input parameter values, $\{ \boldsymbol \xi_l \}_{l=1}^r$, defines the low-rank basis used to reconstruct the full data matrix. From (\ref{eq:id2}), the bi-fidelity approximation of $\boldsymbol v_H(\boldsymbol\xi_j)$ is obtained by evaluating the high-fidelity model at the parametric inputs identified in the basis, $\{ \boldsymbol \xi_l \}_{l=1}^r$, to form 
\begin{equation}
\label{eq:hr}
    \bm H_r := \left[ \boldsymbol v_H(\boldsymbol\xi_1) \ \boldsymbol v_H(\boldsymbol\xi_2) \ \cdots \ \boldsymbol v_H(\boldsymbol\xi_r) \right].
\end{equation}
Then, the coefficients computed from the low-fidelity dataset are used to reconstruct the full high-fidelity data matrix:
\begin{equation}
    \bm H \approx \widehat{\bm H} := \bm H_r \bm C_L,
\end{equation}
thus obtaining predicted high-fidelity realizations at the $M - r$ parameter instances for which only the low-fidelity model was evaluated. An {\it a posteriori} error estimate for the bi-fidelity solution $ \widehat{\bm H}$ along with conditions on the pair of low- and high-fidelity data for an accurate $\widehat{\bm H}$ are presented in \cite{hampton2018}.

\subsection{Multi-modal data driven by stochasticity}
\label{}

We now consider a case for the QoI $\boldsymbol v(\boldsymbol\xi)$ specified by the following two assumptions:
\begin{assumption}
    For a fixed $\bm\xi$, $\boldsymbol v(\boldsymbol\xi)$ is a multivariate random vector. That is, in addition to $\bm\xi$, $\boldsymbol v$ depends on some random variables $\bm\omega$ which may be hidden (unknown).\\[-.2cm]
\end{assumption}

\begin{assumption}
    The joint distribution of $\boldsymbol v(\boldsymbol\xi)$ conditioned on $\bm\xi$ is multi-modal. Stated differently, for a given $\bm\omega$, the realizations of $\boldsymbol v(\boldsymbol\xi)$ are clustered. This implies that the joint probability density of $\boldsymbol v(\boldsymbol\xi)$ conditioned on $\bm\xi$, denoted by $f(\bm v | \bm \xi)$, can be written as a mixture of component densities
\begin{equation}
\label{eq:mixturedist}
 f(\bm v | \bm \xi) = \sum_{k=1}^K \pi_k(\bm \xi) f_k(\bm v; \bm \xi), \quad \sum_{k=1}^K \pi_k(\bm \xi)=1,
\end{equation}
where $f_k(\bm v; \bm \xi)$ is the probability density of the $k$-th mode (cluster), parameterized by $\bm \xi$. 
\end{assumption}

Reintroducing the low-rank approximation of (\ref{eq:id}) results in
\begin{equation}
    \label{eq:randomid}
    \boldsymbol v(\boldsymbol\xi, \bm \omega) \approx \sum_{l=1}^r \boldsymbol v(\boldsymbol\xi_l, \bm \omega_l) c_l(\boldsymbol\xi, \bm \omega),
\end{equation}
where, as before, the choice of basis input samples $\bm \xi_l$ and interpolation weights $c_L$ are implicitly dependent on a collection of training data. In the bi-fidelity case, $\bm v_L(\bm \xi)$ and $\bm v_H(\bm \xi)$ are independent multivariate random variables and are assumed to be distributed according to a probability density of the form in (\ref{eq:mixturedist}). A naive application of the bi-fidelity procedure described in Section \ref{sec:bf_id} yields
\begin{align}
    & \boldsymbol v_L(\bm \xi, \bm \omega) \approx \sum_{l=1}^r \boldsymbol v_L(\bm \xi_l, \bm \omega_l^L) c_l(\bm \xi, \bm \omega), \\
    & \boldsymbol v_H(\bm \xi, \bm \omega) \approx \sum_{l=1}^r \boldsymbol v_H(\bm \xi_l, \bm \omega_l^H) c_l(\bm \xi, \bm \omega) \label{eq:randomid2}.
\end{align}
However, because $\bm \omega_l^L$ and $\bm \omega_l^H$ are independent realizations, they may have different cluster membership, meaning that re-using identical $c_l$ in the approximation of (\ref{eq:randomid2}) may be inappropriate. The basic assumption of our method is as follows.
\begin{assumption}
    If the cluster membership of $\bm \omega_l^L$ matches that of $\bm \omega_l^H$, then (\ref{eq:randomid2}) still holds.
\end{assumption}

In this work, we propose a procedure whereby $\bm \omega_l^L$ and $\bm \omega_l^H$ are guaranteed to have matching cluster membership. Our algorithm, denoted as the \textit{multi-modal ID}, is described next. Although the method can tolerate a small difference in cluster probability between fidelities, the final bi-fidelity approximation uses only samples from the low-fidelity model to infer cluster probabilities used to produce predictions.

\subsection{Multi-modal ID method}
\label{sec:mmidmethod}

In the standard bi-fidelity ID, the low-fidelity model is evaluated at $N$ input samples $\{ \boldsymbol \xi_j \}_{j=1}^N$. In the stochastic multi-modal case, we sample $\bm v_L(\bm\xi_j)$ input $N_S$ times, giving a total of $N \cdot N_S$ low-fidelity evaluations. $N_S$ can be chosen to sufficiently sample the distribution of $\boldsymbol v_L(\boldsymbol\xi_j)$. We now have a set of $N_S$ realizations of the $M_L \times N$ low-fidelity data matrix. The modified ID approach consists of three steps: basis selection, sampling of high-fidelity, and computation of the interpolation weights. An algorithmic description of the multi-modal ID method is provided in Algorithm \ref{alg:1}, Appendix \ref{app:stochid}.

\paragraph{Basis selection} 

In the baseline ID method, the basis is found via a column-pivoted rank-revealing QR factorization of $\bm L$, $\bm L \bm P \approx \bm Q_L \bm R_L$. Then the basis column index set
\begin{equation}
\label{eq:J}
    \mathcal{J} = \left\{ \left. \argmax_i P_{ij} \right \rvert j = 1, 2, \dots, r \right\}.
\end{equation}

In the stochastic case, $\bm L$ is a random matrix; therefore, the optimal basis is given by minimizing the expected approximation error over the distribution of $\bm L$, i.e.,
\begin{align}
\label{eq:optimalj}
    \mathcal{J} &= \argmin_{\mathcal{J}} \ \mathbb{E}\left[ \| \bm L - \bm L(:, \mathcal{J}) \bm C(\bm L, \mathcal{J}) \|_F^2 \right], \\ \bm C(\bm L, \mathcal{J}) &= \argmin_{\bm C} \| \bm L - \bm L(:, \mathcal{J}) \bm C \|_F^2, \label{eq:clstsq}
\end{align}
where $\bm L(:, \mathcal{J})$ is the basis and $\bm C(\bm L, \mathcal{J})$ is the interpolation matrix obtained via the solution of a matrix least squares problem. Equation (\ref{eq:optimalj}) describes a stochastic combinatorial optimization problem, where the size of the search space is $N \choose r$. We experimented with two basis selection algorithms, which both operate on $N_S$ samples of $\bm L$, with $\bm L^{(i)}$ denoting the $i$-th sample.

\begin{enumerate}
    \item A column-pivoted QR based on vertical stacking of samples of $\bm L$. Then the rank-$r$ column-pivoted QR decomposition is
    \begin{equation}
        \begin{bmatrix}
             \bm L^{(1)} \\
             \bm L^{(2)} \\
             \vdots \\ 
             \bm L^{(N_S)}
        \end{bmatrix} \bm P \approx \bm Q_L \bm R_L.
    \end{equation}
    The basis index set is given by (\ref{eq:J}) as before. Recall that $\bm L^{(i)} \in \mathbb{R}^{M_L \times N}$, and, thus, the complexity of this approach is $\mathcal{O} \left(r \cdot N_S \cdot M_L \cdot N \right)$.

\item A sample average approximation (SAA) approach \cite{kleywegt2002sample} in which (\ref{eq:optimalj}) is converted to the SAA problem
\begin{equation}
\label{eq:saa}
     \argmin_{\mathcal{J}} \ \frac{1}{N_S} \sum_{i=1}^{N_S} \| \bm L^{(i)} - \bm L^{(i)}(:, \mathcal{J}) \bm C^{(i)}(\bm L^{(i)}, \mathcal{J}) \|_F^2.
\end{equation}
The problem in (\ref{eq:saa}) is still intractable, and due to the non-convexity of the objective function, we solve this optimization problem numerically using a simulated annealing algorithm \cite{van1987simulated}. 

The cost function (denoted by $\ell(\mathcal{J})$) is given by the right hand side of (\ref{eq:saa}), with $\bm C(\bm L, \mathcal{J})$ the solution of the matrix least squares problem in  (\ref{eq:clstsq}). The search space is explored by swapping a single randomly-chosen column index at iteration $k$, and then accepting the change with probability
\[
\begin{cases}
1, & \text{if } \ell(\mathcal{J}_{k+1}) < \ell(\mathcal{J}_{k}),\\[4pt]
\exp\!\left(-\dfrac{\ell(\mathcal{J}_{k+1}) - \ell(\mathcal{J}_{k})}{T_{k}}\right), & \text{otherwise},
\end{cases}
\]
where $T_{k}$ is a temperature parameter which varies according to a preset schedule. This procedure is repeated for multiple random initial guesses. An additional regularization heuristic used here is that synthetic samples of $\bm L$ are bootstrapped by randomly combining columns from the original samples (because columns are independently distributed), leading to the cost function
\begin{equation}
    \ell(\mathcal{J}) = \frac{1}{N_b} \sum_{i=1}^{N_b} \| \bm L^{(i)} - \bm L^{(i)}(:, \mathcal{J}) \bm C^{(i)}(\bm L^{(i)}, \mathcal{J}) \|_F^2.
\end{equation}

The simulated annealing algorithm introduces three new parameters which affect the complexity: the number of iterations $N_{\rm iter}$, the number of restarts $N_{\rm restart}$, and the number of bootstrap samples $N_b$. The overall complexity of the algorithm is $\mathcal{O} \left( N_{\rm iter} \cdot N_{\rm restart} \cdot N_b \cdot (r^2 \cdot M_L + r \cdot M_L \cdot N + r^2 \cdot N )\right)$, because the inner matrix least squares solution is $\mathcal{O} \left( r^2 \cdot M_L + r \cdot M_L \cdot N + r^2 \cdot N \right)$ using a QR decomposition of $\bm L(:, \mathcal{J})$. However, because in general $N_{\rm iter} \propto N$ the simulated annealing algorithm is less scalable than column stacked QR.

\end{enumerate}

The two approaches are compared numerically in Section \ref{sec:exp:basis}. Unless otherwise specified, we use the simulated annealing algorithm for basis selection.

\paragraph{Sample the high-fidelity} This step is unchanged in the multi-modal version of the algorithm. As before, the high-fidelity model is evaluated at only $r$ parametric inputs $\{ \boldsymbol \xi_l \}_{l \in \mathcal{J}}$. These samples are assembled into the $M_H \times r$ matrix $\boldsymbol H_r$ (restating equation \ref{eq:hr} using the notation of $\mathcal{J}$ for the basis column index set): 
\begin{equation}
    \boldsymbol H_r = \left[ \bm v(\bm \xi_{\mathcal{J}[1]}) \ \bm v(\bm \xi_{\mathcal{J}[2]}) \ \cdots \ \bm v(\bm \xi_{\mathcal{J}[r]})\right].
\end{equation}

\paragraph{Compute interpolation weights} The interpolation weights $\bm C_L$ are combined with the high-fidelity basis to obtain the final bi-fidelity approximation $\widehat{\bm H} = \bm H_r \bm C_L$. A single sample of $\bm C_L$ is obtained by first sampling from $\bm L$, conditioned on the basis columns agreeing with the cluster membership of the high-fidelity basis. This sample is obtained by:
\begin{enumerate}
    \item For each of the non basis columns ($j\notin \mathcal{J}$), randomly select one of the $N_S$ samples of this column.
    \item For each basis column ($j \in \mathcal{J}$), randomly select a sample of this column from samples where the cluster membership agrees with the cluster membership of the corresponding high-fidelity basis column.
\end{enumerate}
Denoting this sample as $\tilde{\bm L}$, then a sample of the interpolation weight matrix is obtained through the solution of the regularized least squares problem:
\begin{equation}
    \argmin_{ \bm C_L} \| \tilde{\bm L} - \tilde{\bm L}(:, \mathcal{J}) \bm C_L \|_F^2 + \lambda \| \bm C_L \|_F^2.
\end{equation}
This step is repeated, providing multiple samples of $\widehat{\bm H}$. The variation in these samples reflects the variation in the interpolation rule, due to the variation in $\bm L$. The parameter $\lambda$ is set to 1, with data scaled to $\mathcal{O}(10)$ in all numerical experiments.

\subsection{Illustrative toy example}
\label{sec:failure}

\begin{figure}[b!]
    \centering
    \includegraphics[width=\linewidth]{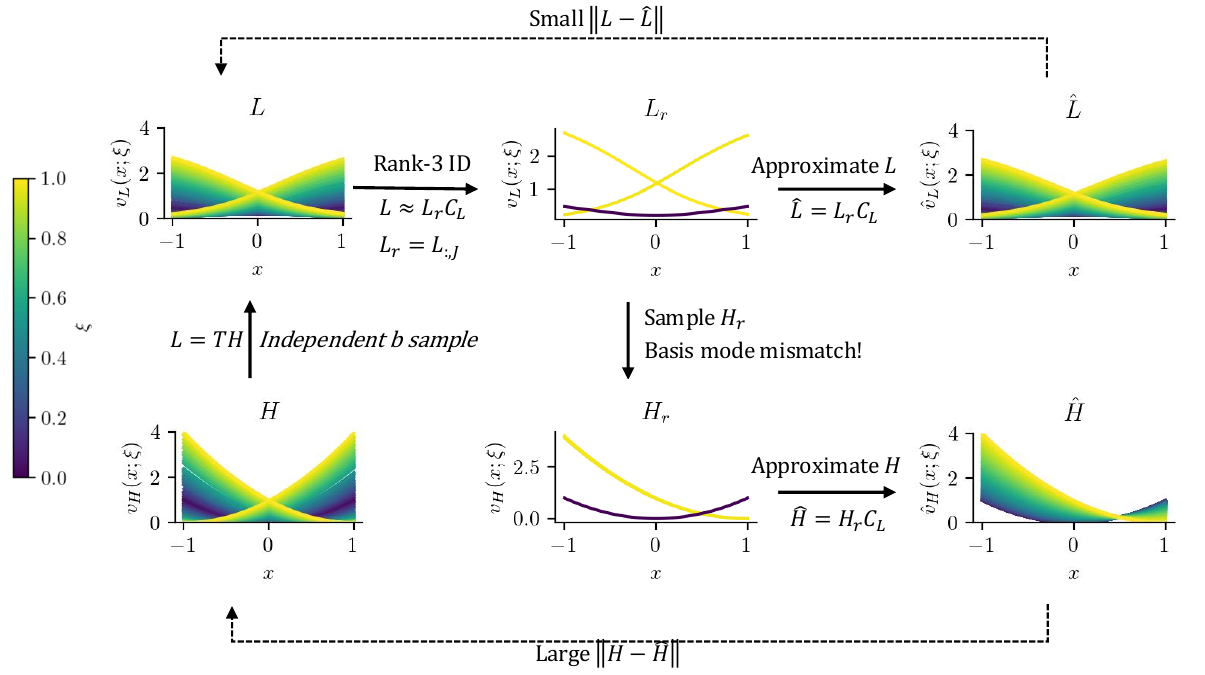}
    \caption{Toy example illustrating the failure of the baseline bi-fidelity ID method on stochastic multi-modal data. The true data matrices $\bm L$ and $\bm H$ are illustrated on the left. In practice, we only observe $\bm L$. The rank-3 ID of $\bm L$ identifies the 3 basis columns illustrated in $\bm L_r$. When $\bm H_r$ is sampled, the basis modes do not match (two of the $v_H(x;\xi)$ samples in $\bm H_r$ are overlapping), leading to the incorrect prediction of $\widehat{\bm H}$ despite the accurate approximation of $\bm L$.
    }   
    \label{fig:toy}
\end{figure}

We now illustrate the failure of the standard bi-fidelity ID using a toy example. Each realization of the high-fidelity QoI, $v_H(x; \xi)$, follows the simple bi-modal discrete distribution with quadratic dependence on a single scalar parameter $\xi$:
\begin{equation}
v_H(x; \xi, \omega) = 
\begin{cases}
(x - \xi)^2 & \text{if } \omega = 0 \\
(x + \xi)^2 & \text{if } \omega = 1,
\end{cases}
\label{eq:toy}
\end{equation}
where \( \omega \sim \text{Bernoulli}(\frac{1}{2}) \). We draw 100 samples of $\xi$ following $U(0, 1)$, one for each column of $\bm H$, and evaluate $x$ at 100 equi-spaced points in $[0,1]$. In this case, $\mathrm{rank}(\bm H) = 3$. Also, note that in practice our method does not require any knowledge of the random variable $\omega$ which drives the multi-modality.

The low-fidelity data is obtained by an independent sampling of $\omega$, followed by convolution with a Gaussian kernel with $\sigma=\frac{1}{2}$, represented by the linear operator $\bm T$ defined as
\begin{equation}
\label{eq:T}
    T_{ij} = \frac{1}{Z_i} \exp\left( -\frac{(x_i - x_j)^2}{2 \sigma^2} \right),
\end{equation}
where $Z_i$ is the row-normalization factor.
Since $\bm T$ is invertible, in the absence of stochasticity due to the independent sampling of $\omega$, the bi-fidelity approximation of $\bm H$ is exact. See Appendix \ref{app:linearop} for further remarks on the accuracy of bi-fidelity ID when $\bm L=\bm T \bm H$. 

However, Figure \ref{fig:toy} illustrates how this procedure fails in the face of a bi-modal data distribution due to the independent sampling of $\omega$. The key point at which the method breaks down is in the sampling of the high-fidelity basis $\bm H_r$. These columns are sampled independently of the low-fidelity basis, and if their cluster membership is not identical, then the interpolation rule $\bm C_L$ will not be appropriate.  

Furthermore, the probability of a basis mismatch converges to 1 as $r$ increases under an assumption on the mixture distribution.

\begin{proposition}
For mixture distributions with constant mixture weights 
\begin{align}
    & f_L(\bm v_L | \bm \xi) = \sum_{k=1}^K \pi_k f_{L,k}(\bm v_L; \bm \xi), \\
    & f_H(\bm v_H | \bm \xi) = \sum_{k=1}^K \pi_k f_{H,k}(\bm v_H; \bm \xi),
\end{align}
with $\sum_{k=1}^K\pi_k=1$, the probability of basis mismatch, $P(MM)$, is bounded from below by $1 - (\max_k \pi_k)^r$ and bounded from above by $1 - (\min_k \pi_k)^r$.
\end{proposition}

\begin{proof}
    $P(MM) = 1 - P(MM^C)$, i.e., the probability of a basis vector mismatch is 1 - (probability of the basis matching). Each basis vector, indexed from $1$ to $r$, is independent. Call the probability of the $i$-th basis element matching $P(M_i)$. Then $P(MM^C) = \prod_{i=1}^rP(M_i)$.
    In a mixture with constant weights $\pi_k$, the probability that a basis element matches is bounded by $\min_k \pi_k \leq P(M_i) \leq \max_k \pi_k$. Thus, $(\min_k \pi_k)^r \leq P(MM^C) \leq (\max_k \pi_k)^r$. Taking complements gives $1 - (\max_k \pi_k)^r \leq P(MM) \leq 1 - (\min_k \pi_k)^r$.
\end{proof}

Basis mismatch is therefore highly likely as $r$ increases for a uniform mixture distribution.

\subsection{Convergence of binary mixture probability}

We now show the convergence rate of the predicted mixture probability obtained with the multi-modal ID method for the case of a binary mixture ($K=2$). This is a property of nested binomial processes, as explained below.
To simplify the exposition, consider a single non-basis column $i$ whose true mixture probability is denoted by $\pi$ (so that $\pi_1 + \pi_2 = 1$ and $\pi = \pi_1$). We generate $N_S$ samples of this column under the low-fidelity model and record the success proportion 
\[
    p_1 \;\sim\; \frac{\mathrm{Binomial}(N_S, \pi)}{N_S}.
\]
Then, using the procedure of Section~\ref{sec:mmidmethod}, we create $S$ realizations of the predicted high-fidelity data matrix $\widehat{\bm{H}}$. The mixture probability in these $S$ realizations is observed as
\[
    \hat{\pi} \;\sim\; \frac{\mathrm{Binomial}(S, p_1)}{S}.
\]
This nesting of binomial processes allows us to make the following statement about the unbiasedness of $\hat{\pi}$ and the variance of our estimator.

\begin{proposition}

The multi-modal ID estimator $\hat{\pi}$ is unbiased, i.e., $\mathbb{E}[\hat{\pi}] \;=\; \pi$,
and its variance is $\mathrm{Var}(\hat{\pi}) = \pi\,(1-\pi)\,\frac{N_S + S - 1}{N_S\,S}$.
Consequently, $\mathrm{Var}(\hat{\pi}) \to 0$ as $N_S, S \to \infty$.
\end{proposition}

\begin{proof}
First, using the law of total expectation:
\begin{equation}
    \mathbb{E}[\hat{\pi}] = \mathbb{E}[\mathbb{E}[\hat{\pi} | p_1]] = \pi.
\end{equation}
And similarly, the law of total variation:
\begin{align}
    \textrm{Var}(\hat{\pi}) =& \ \mathbb{E}[\textrm{Var}(\hat{\pi} | p_1)] + \textrm{Var}(\mathbb{E}[\hat{\pi} | p_1]) \\
     =& \ \mathbb{E} \left[ \frac{p_1 (1 - p_1)}{S} \right] + \textrm{Var}(p_1) \\ 
     =& \ \pi (1 - \pi) \frac{N_S + S -1}{N_S S}.
     \label{eq:var_pi}
\end{align}

As $N_S, S \to \infty$, this variance tends to zero. In the special case $N_S = S$, the variance scales on the order of $1/S$.
\end{proof}

\section{Experiments and Discussion}
\label{sec:experiments}

In the following section, we present numerical experiments in which we quantify various aspects of the proposed method. Synthetic problems are introduced in Section \ref{sec:toy_probs}. The two basis selection approaches which we have developed for the multi-modal ID method are then quantified in Section \ref{sec:exp:basis}. The performance of the multi-modal and baseline ID methods is compared in Section \ref{sec:exp:comparison}. Finally, the proposed method is applied to sensitivity analysis in a stochastic multi-physics problem in Section \ref{sec:exp:combustor}.

\subsection{Synthetic Toy Problems}
\label{sec:toy_probs}

We compare the multi-modal ID method against the baseline ID method on a series of three synthetic problems with progressively increasing complexity.
In each toy problem, the number of low-fidelity data matrix samples, $N_S$, is set to 10. 

\paragraph{A. Discrete Distribution with Quadratic Dependence on $\xi$}
This problem serves as a pure test of the algorithm's ability to handle stochasticity that leads to a bi-modal distribution, isolated from low-rank approximation errors.
\begin{itemize}
    \item The data is drawn directly from the bimodal distribution in (\ref{eq:toy}), sampling $\omega$ independently for each fidelity:
    \begin{align}
        & \bm H = \left[ \bm v_H(\xi_1, \omega_1) \ \cdots \ \bm v_H(\xi_{N}, \omega_{N}) \right], \\
        & \bm L = \bm T \left[ \bm v_H(\xi_1, \omega_1') \ \cdots \ \bm v_H(\xi_{N}, \omega_{N}') \right].
    \end{align}
    \item The low-fidelity model is given by the convolution of the high-fidelity model with a Gaussian kernel, as specified in (\ref{eq:T}).
    \item A 1D parameter $\xi$ is sampled uniformly in $[0,1]$ to generate the 100 columns of the data matrix.
\end{itemize}

\paragraph{B. Pitchfork Bifurcation ODE}
This is a classic bifurcating dynamical system.
\begin{itemize}
    \item The initial value problem is given by:
    \begin{equation}
        \frac{dx}{dt} = \xi x - x^3, \quad x(0) \sim 10^{-3} \times \textrm{Uniform}\{-1, 1\},
    \end{equation}
    This sends trajectories to one of two stable equilibria for $\xi > 0$ with equal probability.
    \item High-fidelity trajectories were obtained using the RK5(4) method \cite{dormand1980family} with relative error tolerance of $10^{-6}$. Low-fidelity trajectories were obtained using the RK3(2) method with relative error tolerance of $10^{-2}$.
    \item The bifurcation parameter $\xi$ is sampled uniformly in $[0, 5]$.
\end{itemize}

\paragraph{C. Stochastic Lotka-Volterra ODE}
This problem represents a more complex, multi-parameter scenario. The Lotka-Volterra equations are an idealized model of predator-prey interactions. A stochastic bi-modal outcome is introduced by sampling the four parameters of the ODE system $(\alpha, K, \beta, \gamma)$ from a two-component mixture distribution. 
\begin{itemize}
    \item The initial value problem is given by
    \begin{align}
\frac{dx}{dt} =& \alpha x \left(1 - \frac{x}{K}\right) - \beta x y \\ 
\frac{dy}{dt} =& \beta x y - \gamma y, \\
x(0), & \ y(0) = 0.5,
\end{align}
where $x$ and $y$ represent the prey and predator populations, respectively. The parameters $\alpha$, $K$, $\beta$, and $\gamma$ are dependent on $\bm\xi$ and are sampled from the discrete mixture distribution
\begin{equation}
\label{eqn:mix_lotka}
\begin{pmatrix}
    \alpha \\
    K \\
    \beta \\
    \gamma
\end{pmatrix}
\sim \phi(\bm\xi)
\begin{pmatrix}
    0.5 \\
    5 \\
    2 + 0.5(\xi_1 + \xi_2) \\
    1
\end{pmatrix}
+ (1 - \phi(\bm\xi))
\begin{pmatrix}
    0.5 \\
    10 \\
    0.25(\xi_1 + \xi_2) \\
    1
\end{pmatrix}.
\end{equation}
In \eqref{eqn:mix_lotka}, the mixture function $\phi(\bm\xi)$ is defined as a sigmoid transformation of a hyperplane in the $\bm\xi$-space: $ \phi(\bm\xi) = \frac{1}{1 + \bm n \cdot \bm\xi - m}$,
where $\bm n = [10, 10] / \sqrt{2}$, and $m=5\sqrt{2} - 0.5$. The model depends on input parameters $\boldsymbol{\xi} = (\xi_1, \xi_2)$ sampled uniformly in the unit square $[0,1] \times [0,1]$. This system can be interpreted as predator-prey dynamics across a randomly varying landscape.

    \item The high-fidelity solution is obtained by a numerical solution up to time $t=50$ using the RK5(4) method with relative error tolerance of $10^{-3}$. The low-fidelity solution is obtained by an independent sampling of the parameters, followed by a convolution of the high-fidelity solution with a Gaussian profile with standard deviation of $\frac{10}{3}$.
\end{itemize}

\subsection{Basis selection}
\label{sec:exp:basis}

We first present experimental results on the basis selection sub-step of the algorithm.
Figure \ref{fig:exp:basis} shows a comparison between the two basis selection approaches. We measure the scaling of $l_2$ reconstruction error as a function of rank, for different values of $N_L$ (number of samples of the random matrix) and data matrix size (square matrices are sampled for this experiment). The performance of the vertical stacking approach degrades as $N_L$ increases, unlike the simulated annealing approach. The simulated annealing basis selection approach is more accurate for the regime relevant to our applications ($N_L \approx 10$, $r \approx 10$), and is used in the following numerical experiments.

The quadratic bimodal problem was excluded as the data matrix is trivially reconstructed for ranks greater than 3.

\begin{figure}
\begin{subfigure}[b]{0.99\textwidth}
        \centering
        \includegraphics[width=\textwidth]{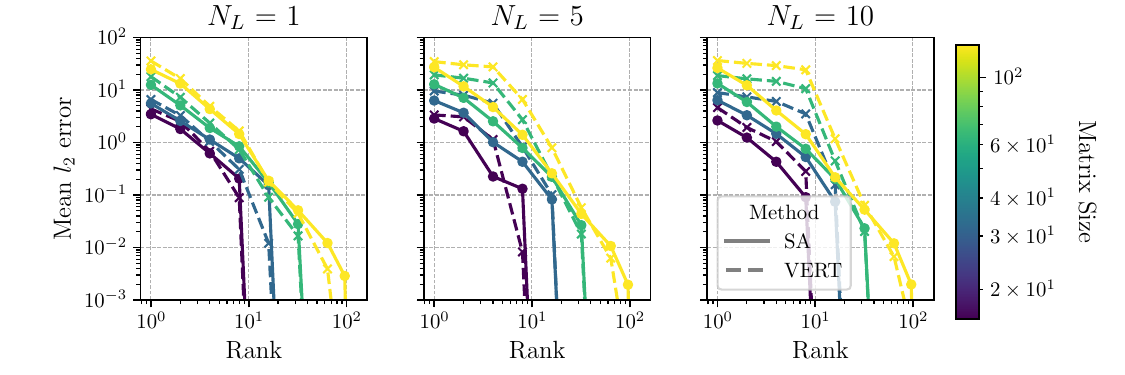}
        \caption{}
        \label{fig:basis:1}
    \end{subfigure}
    
    \begin{subfigure}[b]{0.99\textwidth}
        \centering
        \includegraphics[width=\linewidth]{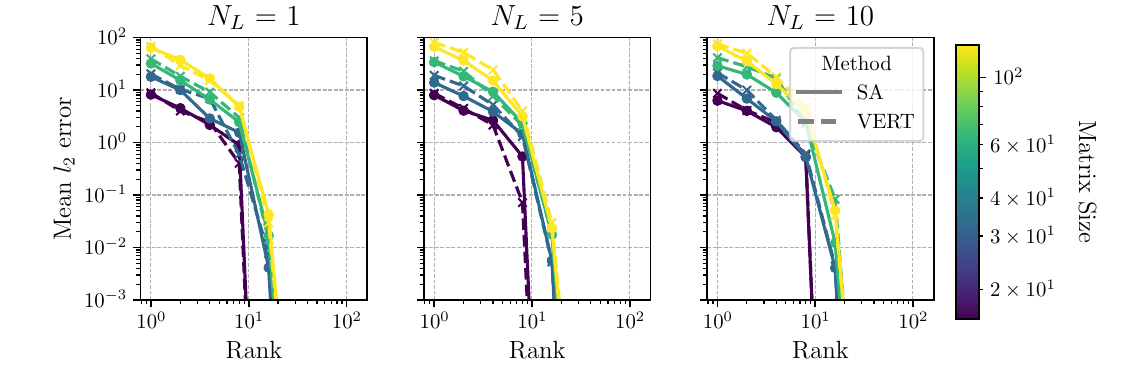}
        \caption{}
        \label{fig:basis:2}
    \end{subfigure}

    \caption{Comparison of proposed basis selection approaches on two synthetic toy problems: \ref{fig:basis:1} pitchfork bifurcation, and \ref{fig:basis:2} spatially varying Lotka-Volterra. Color denotes the size of the data matrix. In the legend, \texttt{SA} is the simulated annealing algorithm, and \texttt{VERT} is the vertically-stacked QR algorithm. }
    \label{fig:exp:basis}
\end{figure}

\subsection{Accuracy comparison on toy problems}
\label{sec:exp:comparison}

We now compare the bi-fidelity approximation accuracy of the multi-modal ID method against the baseline bi-fidelity ID method on the synthetic toy problems defined above.
In all synthetic problems, a rank-5 interpolative decomposition is applied to a $100 \times 100$ data matrix. 

Due to the random nature of the true high-fidelity matrix $\bm H$, the Wasserstein-1 metric is used. We measure $W_1(\bm H, \widehat{\bm H})$ for 50 samples of the flattened matrices $(\bm H, \widehat{\bm H})$ and report the histogram of this distance metric for each method.

The results from the three toy problems are provided in Figures \ref{fig:quadratic}, \ref{fig:pitchfork} and \ref{fig:lv}. Each panel shows the predicted distributions for six randomly selected columns of the high-fidelity data matrix for visualization purposes, along with $W_1(\bm H, \widehat{\bm H})$.

\begin{figure}[htb]
    \begin{subfigure}[b]{0.32\textwidth}
        \centering
        \includegraphics[trim={0 0 3.8cm 0},clip,width=\textwidth]{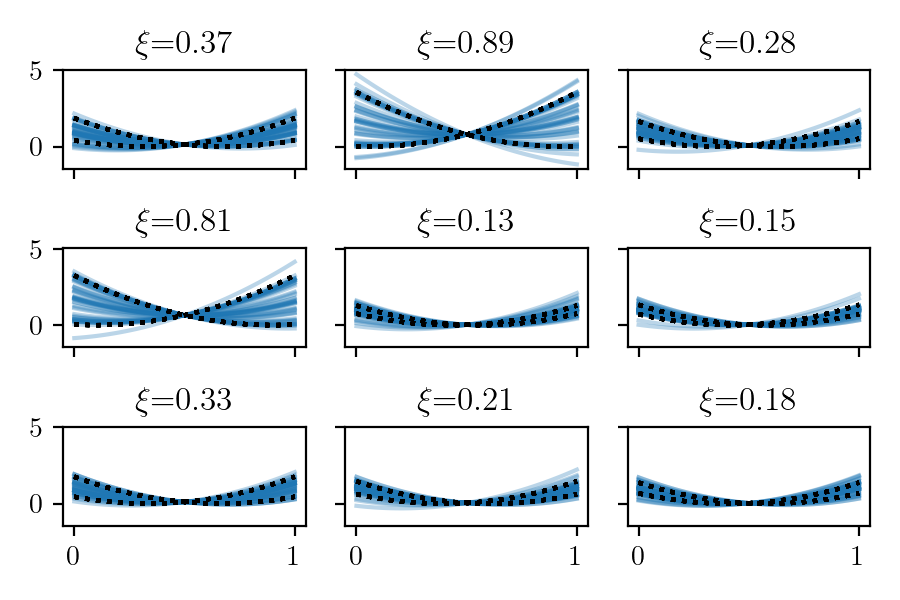}
        \caption{}
        \label{fig:quadratic:detidcols}
    \end{subfigure}
    \begin{subfigure}[b]{0.32\textwidth}
        \centering
        \includegraphics[trim={0 0 3.8cm 0},clip,width=\textwidth]{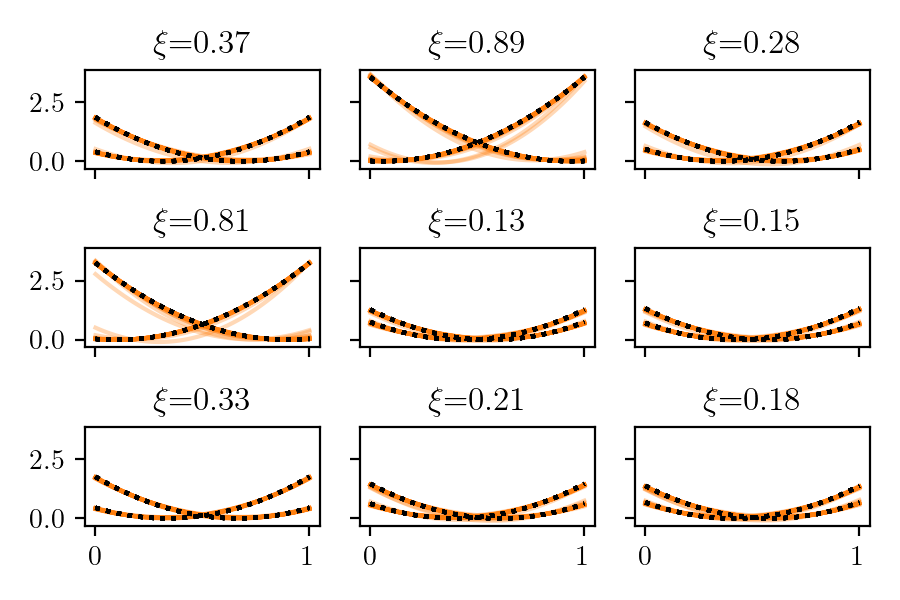}
        \caption{}
        \label{fig:quadratic:stochidcols}
    \end{subfigure}
    \begin{subfigure}[b]{0.32\textwidth}
        \centering
        \includegraphics[width=\linewidth]{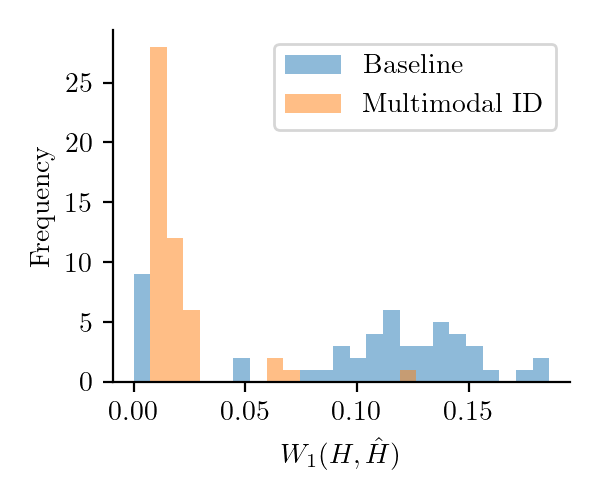}
        \caption{}
        \label{fig:quadratic:w1}
    \end{subfigure}

    \caption{Quadratic bi-modal problem. In \ref{fig:quadratic:detidcols} and \ref{fig:quadratic:stochidcols} a visualization of 100 predictions of six randomly-chosen columns are shown (dotted line = true solution, colored lines = sampled predictions), where each prediction is obtained from an independent sampling of $\bm{L}$ and $\bm{H}_r$. In \ref{fig:quadratic:w1}, the probability of a given average $W_1$ distance is compared between approaches.}
    \label{fig:quadratic}

\end{figure}

\begin{figure}[htb]
    \begin{subfigure}[b]{0.32\textwidth}
        \centering
        \includegraphics[trim={0 0 3.8cm 0},clip,width=\textwidth]{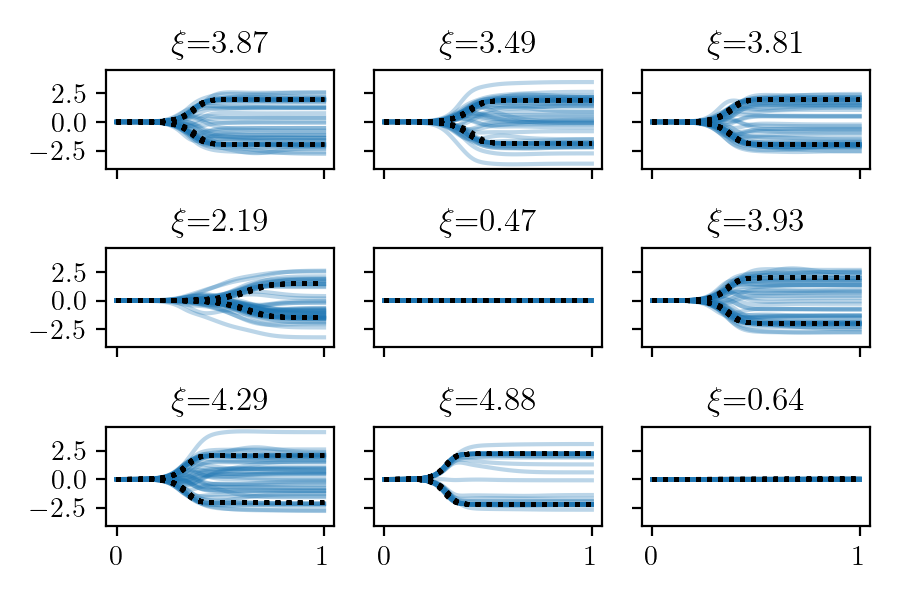}
        \caption{}
        \label{fig:pitchfork:detidcols}
    \end{subfigure}
    \begin{subfigure}[b]{0.32\textwidth}
        \centering
        \includegraphics[trim={0 0 3.8cm 0},clip,width=\textwidth]{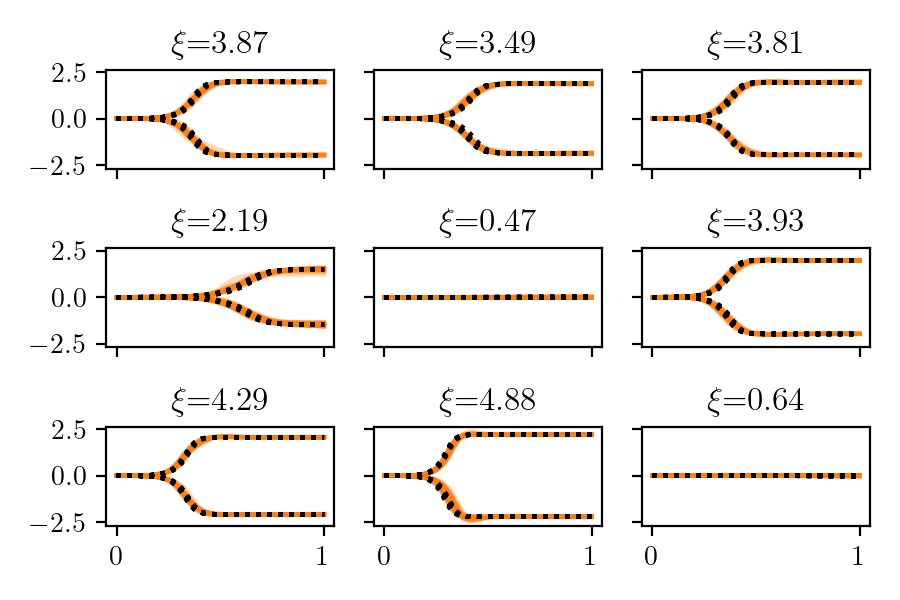}
        \caption{}
        \label{fig:pitchfork:stochidcols}
    \end{subfigure}
    \begin{subfigure}[b]{0.32\textwidth}
        \centering
        \includegraphics[width=\linewidth]{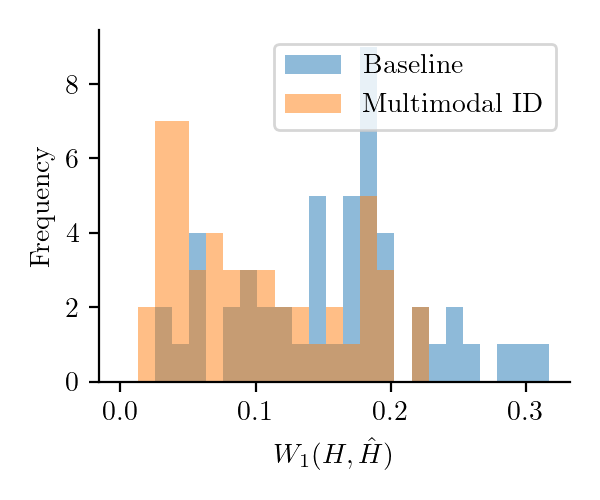}
        \caption{}
        \label{fig:pitchfork:w1}
    \end{subfigure}

    \caption{Pitchfork bifurcation ODE. In \ref{fig:pitchfork:detidcols} and \ref{fig:pitchfork:stochidcols} a visualization of 100 predictions of six randomly-chosen columns are shown (dotted line = true solution, colored lines = sampled predictions), where each prediction is obtained from an independent sampling of $\bm{L}$ and $\bm{H}_r$. In \ref{fig:pitchfork:w1}, the probability of a given average $W_1$ distance is compared between approaches.}
    \label{fig:pitchfork}

\end{figure}

\begin{figure}[htb]
\centering
    \begin{subfigure}[b]{0.32\textwidth}
        \centering
        \includegraphics[trim={0 0 3.8cm 0},clip,width=\textwidth]{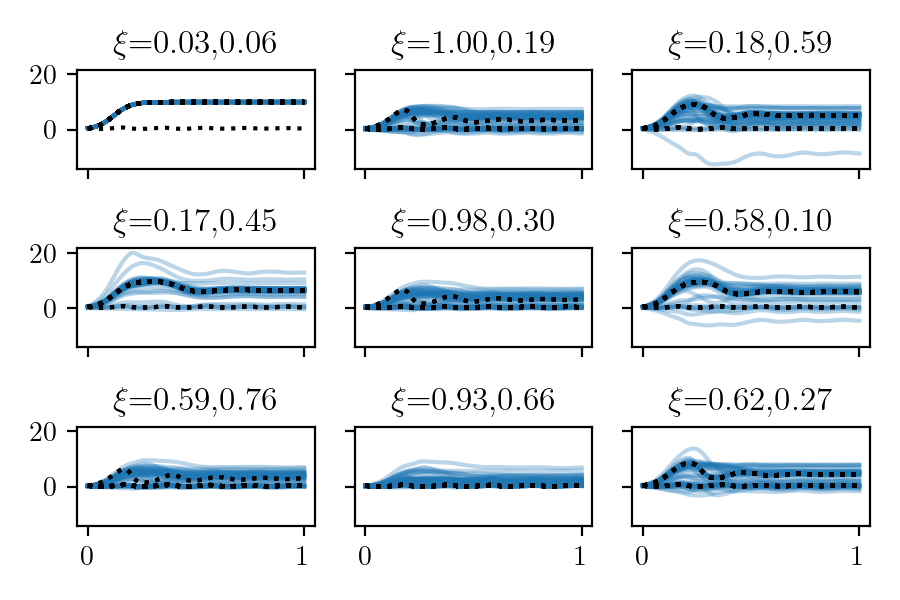}
        \caption{}
        \label{fig:lv:detidcols}
    \end{subfigure}
    \begin{subfigure}[b]{0.32\textwidth}
        \centering
        \includegraphics[trim={0 0 3.8cm 0},clip,width=\textwidth]{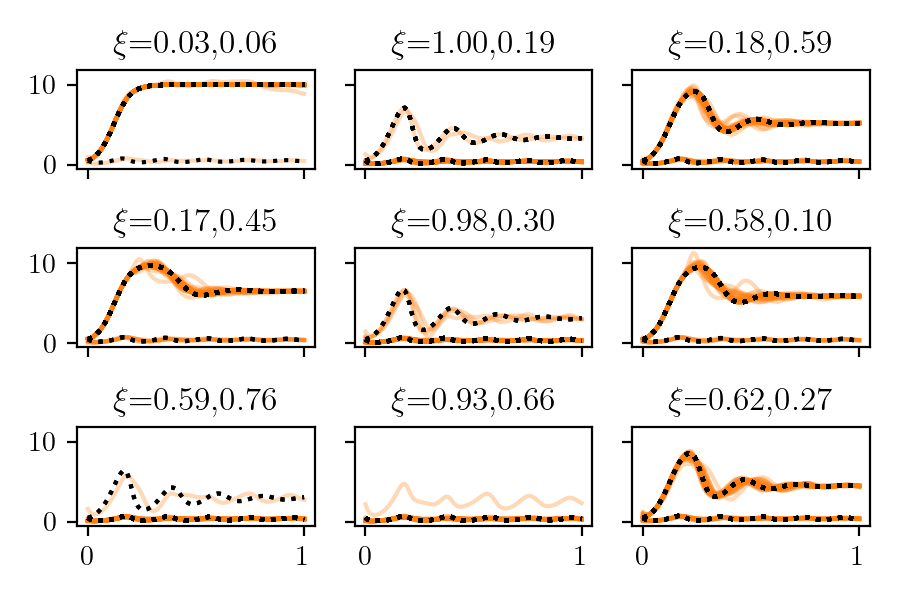}
        \caption{}
        \label{fig:lv:stochidcols}
    \end{subfigure}
    \begin{subfigure}[b]{0.32\textwidth}
        \centering
        \includegraphics[width=\linewidth]{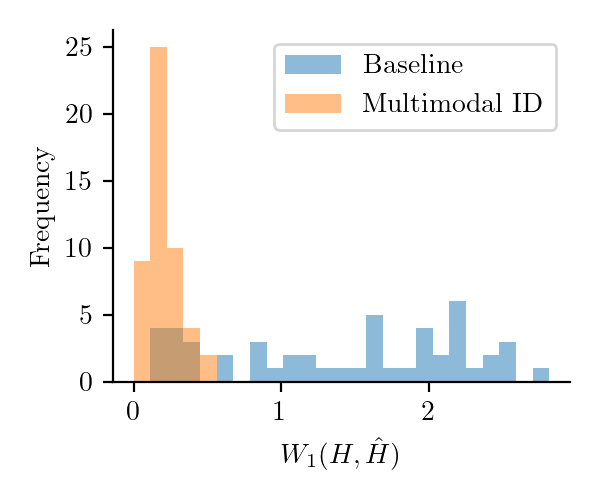}
        \caption{}
        \label{fig:lv:w1}
    \end{subfigure}

    \caption{Stochastic Lotka-Volterra ODE. In \ref{fig:lv:detidcols} and \ref{fig:lv:stochidcols} a visualization of 100 predictions of six randomly-chosen columns are shown (dotted line = true solution, colored lines = sampled predictions), where each prediction is obtained from an independent sampling of $\bm{L}$ and $\bm{H}_r$. In \ref{fig:lv:w1}, the probability of a given average $W_1$ distance is compared between approaches.}
    \label{fig:lv}

\end{figure}

\paragraph{Discussion of synthetic problems}

All three synthetic problems exhibit similar results. The ID algorithm incorrectly blends modes of the solutions, while the multi-modal ID algorithm correctly distinguishes between solution branches, even if they are imperfectly reconstructed due to the low-rank approximation. The $W_1$ distance plots show that for a given realization of the data, multi-modal ID is consistently able to provide a closer match to the true HF data matrix distribution.

In these problems, distinguishing between modes is trivially achieved by thresholding, and we note that this proposed approach is only possible if, for each fidelity, the cluster/mode membership can be identified, and the same set of clusters exists for both fidelities. Furthermore, in these synthetic problems, the cluster probability is identical across fidelities, which is another assumption made by our method. Small differences in cluster probability are tolerated, however.

\subsection{Bi-fidelity approximation of laser-ignited rocket combustor}
\label{sec:exp:combustor}

\paragraph{Introduction to laser-ignited rocket combustor application}

Laser-induced spark is an ignition method in a rocket combustor that facilitates engine re-ignition throughout a mission. In the present system, a circular jet of oxygen with an annular flow of methane is considered. A laser is focused just outside the shear layer, and this injects sufficient energy to ignite the system. However, run-to-run variabilities in the system arise due to interactions with the turbulent shear layer of the methane-oxygen jet. This has been confirmed experimentally in a subscale rocket combustor (Purdue Propulsion, \cite{strelau2023modes}), where reliability maps have been produced experimentally by repeating the same nominal experiment. This motivates an understanding of the distribution of possible outcomes from the system.

\paragraph{Multi-fidelity methodology}

The Hypersonics Task-based Research (HTR) solver described in \citet{di2020htr} for the compressible chemically reacting Navier-Stokes equations is used to carry out simulations.  The computational details concerning the solver parallelism, simulation domain, boundary conditions, and numerics can be found in previous work as part of PSAAP-III \citep{wang2021progress, passiatore2024computational, zahtila2025bi}. Note that the laser is modeled by an energy source term imposed on the fluid.

Two mesh resolutions are employed in this study. The low-fidelity mesh is composed of 2M grid points, and the high-fidelity counterpart is composed of 15M grid points. Both meshes are illustrated in Figure \ref{fig:mesh}. The mesh resolution differs by approximately a factor of 2. The same sub-grid scale models are employed in both cases \citep{smagorinsky1963general}. A fully resolved direct numerical simulation would require approximately 100 billion grid points by comparison, but satisfactory agreement has been achieved between measured quantities in experiments and output statistics from the present resolution used in high-fidelity simulations. The chemistry in the low-fidelity simulations is simplified to a 5-species, 3-step mechanism. 

Simulations were carried out on a cluster consisting of 4 Nvidia V100 GPUs per node. Each simulation on the 15M mesh took approximately 128 GPU-hours, while the 2M mesh used approximately 8 GPU-hours, a cost ratio of approximately 16.

\begin{figure}[htb]
\centering
    \begin{subfigure}[b]{0.9\textwidth}
        \centering
        \includegraphics[width=\textwidth]{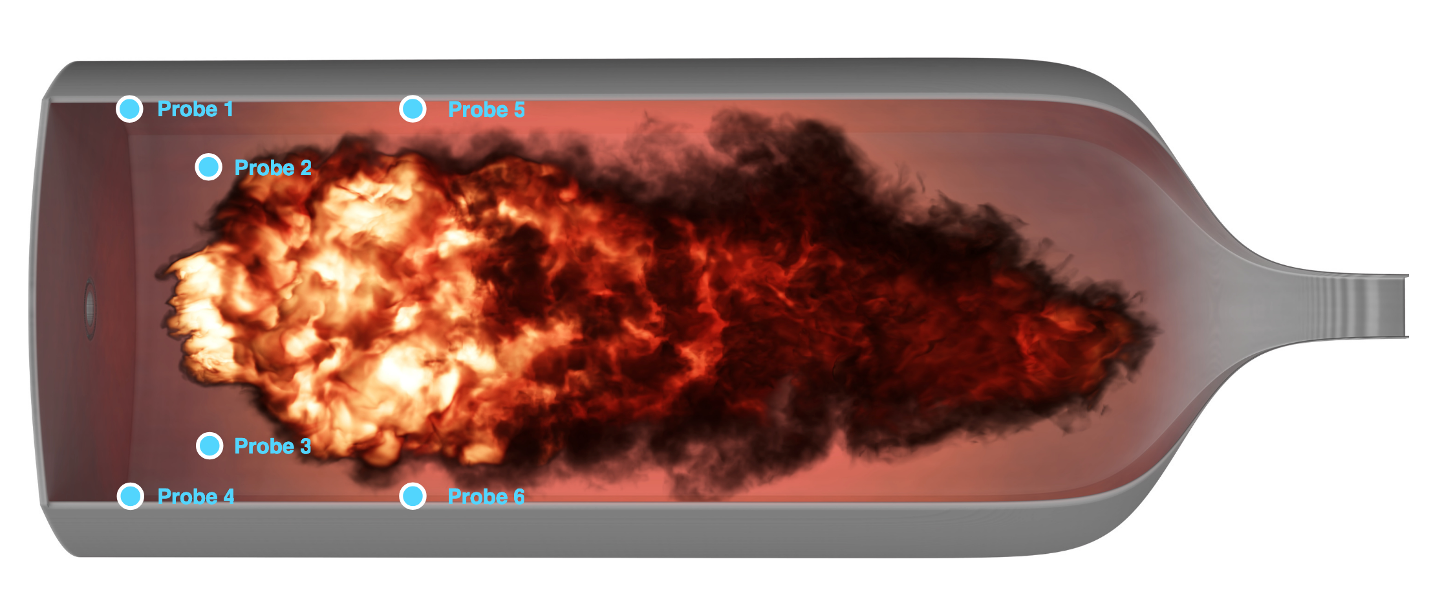}
        \caption{}
        \label{fig:psaap:diego}
    \end{subfigure}
    
    \begin{subfigure}[b]{0.99\textwidth}
        \centering
        \includegraphics[width=\textwidth]{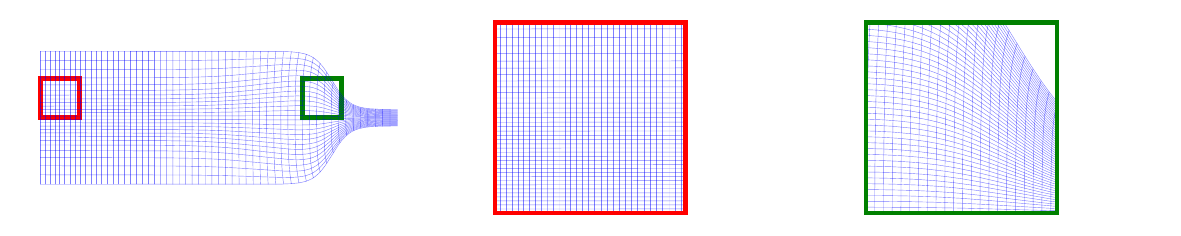}
        \caption{}
        \label{fig:psaap:meshlf}
    \end{subfigure}
    
    \begin{subfigure}[b]{0.99\textwidth}
        \centering
        \includegraphics[width=\linewidth]{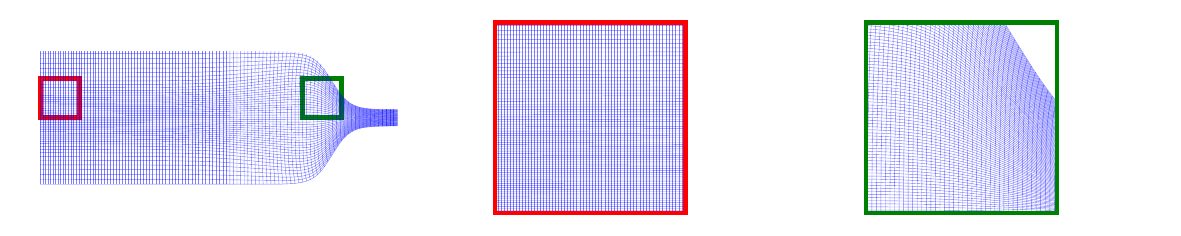}
        \caption{}
        \label{fig:psaap:meshhf}
    \end{subfigure}

    \caption{\ref{fig:psaap:diego} shows a render of temperature in the laser-ignited combustor problem with the location of the six pressure probes. \ref{fig:psaap:meshlf}: illustration of slice from low-fidelity mesh. \ref{fig:psaap:meshhf}: illustration of slice from high-fidelity mesh. In the left part of \ref{fig:psaap:meshlf} and \ref{fig:psaap:meshhf}, the resolution is reduced by a factor of 4 for clarity.}
    \label{fig:mesh}

\end{figure}

\paragraph{Uncertainty quantification problem}

We are interested in quantities associated with the pressure trace obtained from the average of the six probe locations inside the combustor, illustrated in Figure \ref{fig:mesh}. We simulate the pressure over 500 $\mu$s post-laser, and focus on a sensitivity analysis of the average pressure over this time period.

The present uncertainties are informed by both experimentally measured quantities and domain knowledge of the physics that govern the breakdown of the laser deposition kernel. A full and dedicated account of the present uncertainties is given in \citep{zahtilaprogress}. A summary of the resultant uncertainty distributions is presented in Table \ref{tab:uncertainties}, and simulations are run by random sampling of the uncertain inputs $\bm{\xi}$. We run simulations at both fidelities for 237 different samples of our 12-dimensional uncertainty vector. Our objective is the approximation of the high-fidelity data matrix using a bi-fidelity method. This approximated data matrix can then be used in a downstream uncertainty quantification task, such as sensitivity analysis.

Our goal is to compare sensitivity analysis results obtained through the multi-modal ID bi-fidelity approximation with sensitivity analysis results on the full high-fidelity data matrix. We apply three sensitivity indices implemented in SALib \cite{Herman2017}, suitable for pre-existing parameter samples, RBD-FAST \cite{tarantola2006random}, PAWN \cite{pianosi2015simple}, and Delta moment-independent analysis \cite{plischke2013global}.

\begin{figure}[htb]
\centering
    \begin{subfigure}[b]{0.49\textwidth}
        \centering
        \includegraphics[width=\textwidth]{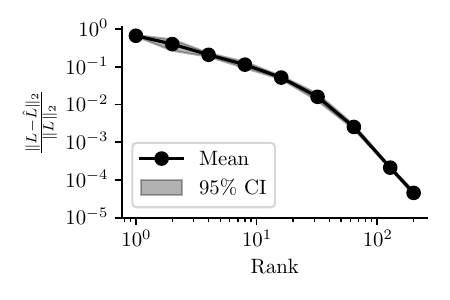}
        \caption{}
        \label{fig:psaap:0}
    \end{subfigure}
    \begin{subfigure}[b]{0.49\textwidth}
        \centering
        \includegraphics[width=\textwidth]{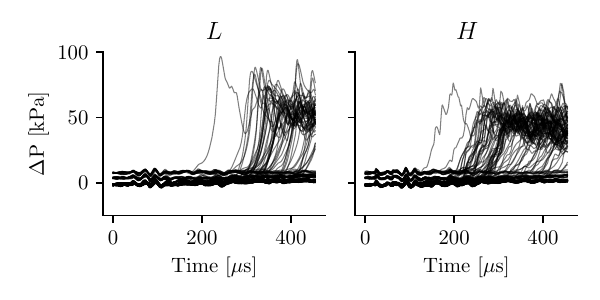}
        \caption{}
        \label{fig:psaap:1}
    \end{subfigure}
    
    \begin{subfigure}[b]{0.49\textwidth}
        \centering
        \includegraphics[width=\textwidth]{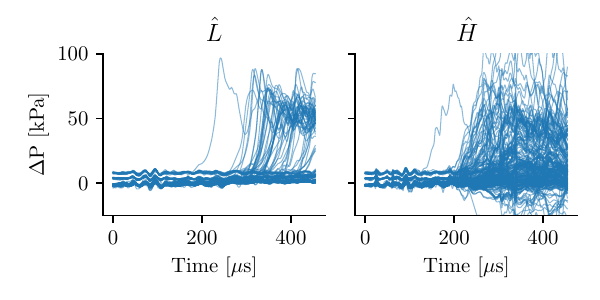}
        \caption{}
        \label{fig:psaap:2}
    \end{subfigure}
    \begin{subfigure}[b]{0.49\textwidth}
        \centering
        \includegraphics[width=\linewidth]{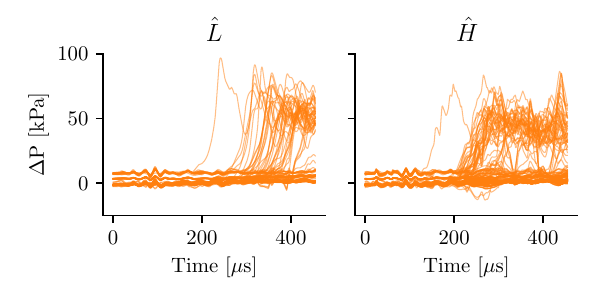}
        \caption{}
        \label{fig:psaap:3}
    \end{subfigure}

    \caption{In \ref{fig:psaap:0} the reconstruction error of the low-fidelity data matrix as a function of rank is shown, with confidence intervals computed from 5 samples of the forward model. The remaining plots illustrate all 237 pressure traces that make up the data matrix. \ref{fig:psaap:1} shows one sample of the true low- and high-fidelity models. \ref{fig:psaap:2} shows the rank 25 approximation of $\bm L$ and $\bm H$ from the baseline ID method, and \ref{fig:psaap:3} shows a realization of the corresponding predictions from the multi-modal ID method. Compare the predictions in \ref{fig:psaap:3} with one sample of the true data matrices in \ref{fig:psaap:1}.}
    \label{fig:psaap_h_hat}
\end{figure}

\begin{figure}[htb]
    \centering
    \includegraphics[width=\textwidth]{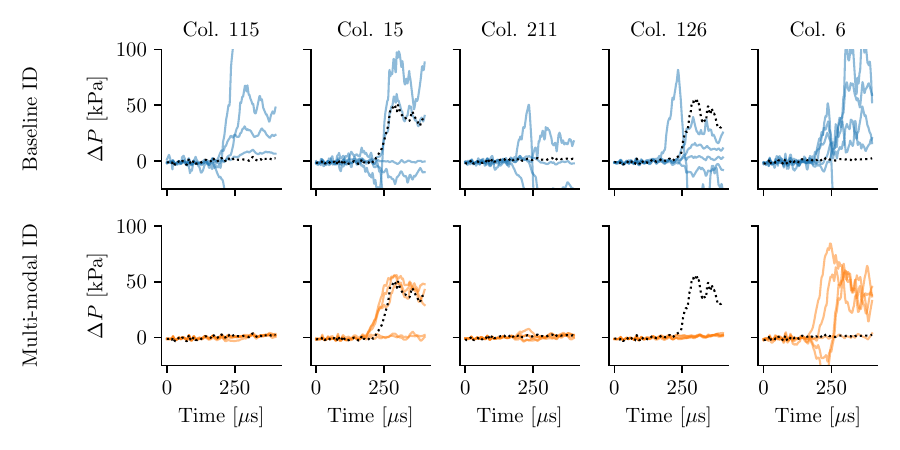}
    \caption{A comparison of predictions from multi-modal ID (bottom row) and baseline ID (top row) for 5 randomly-chosen columns of the data matrix. Five samples are obtained from each method via repeated samples of the low-fidelity model. The black dashed line shows a single realization of the high-fidelity model. The multi-modal ID predictions are generally bimodal, whereas baseline ID fails to give realistic pressure trace predictions.}
    \label{fig:psaap:errobar}
\end{figure}

\begin{figure}[htbp]
\centering
    \begin{subfigure}[b]{0.69\textwidth}
        \centering
        \includegraphics[width=\textwidth]{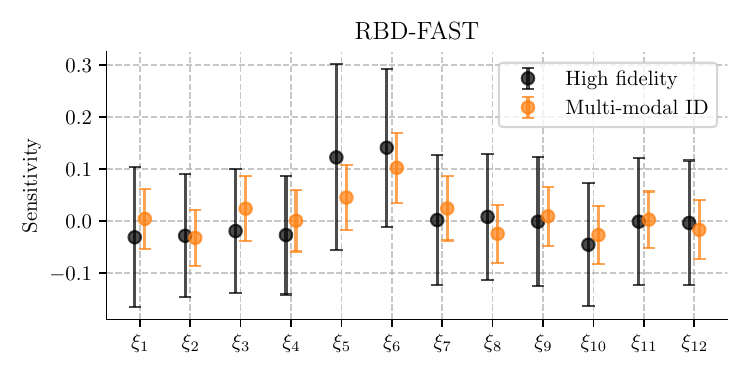}
        \caption{}
        \label{fig:sens:1}
    \end{subfigure}
    \begin{subfigure}[b]{0.29\textwidth}
        \centering
        \includegraphics[width=\textwidth]{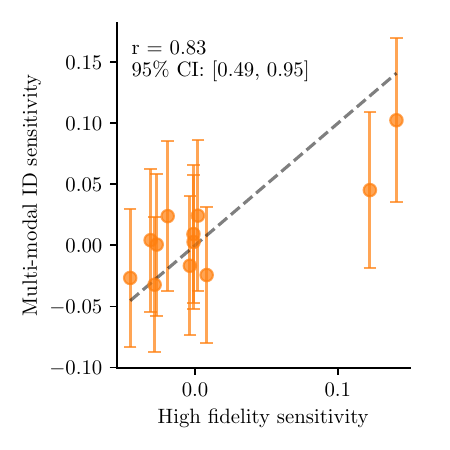}
        \caption{}
        \label{fig:sens:2}
    \end{subfigure}

    \begin{subfigure}[b]{0.69\textwidth}
        \centering
        \includegraphics[width=\textwidth]{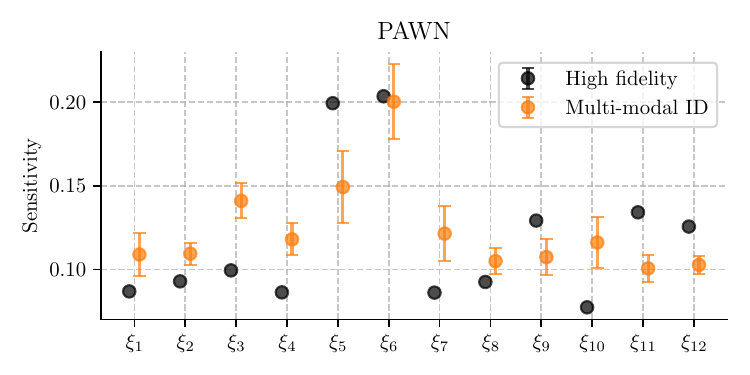}
        \caption{}
        \label{fig:sens:3}
    \end{subfigure}
    \begin{subfigure}[b]{0.29\textwidth}
        \centering
        \includegraphics[width=\textwidth]{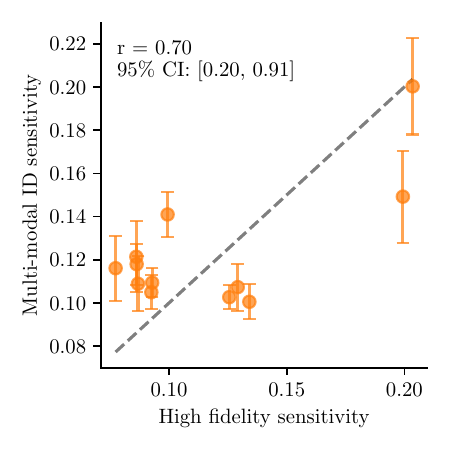}
        \caption{}
        \label{fig:sens:4}
    \end{subfigure}

    \begin{subfigure}[b]{0.69\textwidth}
        \centering
        \includegraphics[width=\textwidth]{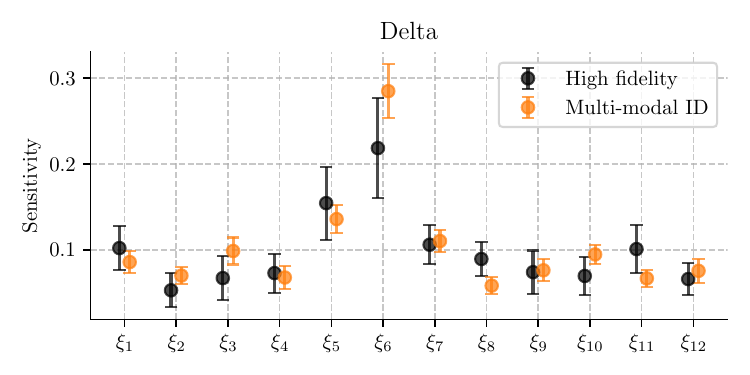}
        \caption{}
        \label{fig:sens:5}
    \end{subfigure}
    \begin{subfigure}[b]{0.29\textwidth}
        \centering
        \includegraphics[width=\textwidth]{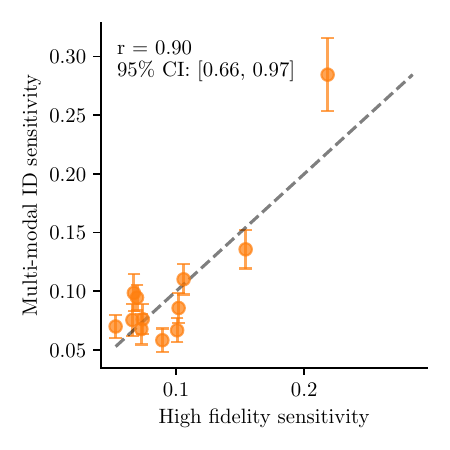}
        \caption{}
        \label{fig:sens:6}
    \end{subfigure}

    \caption{Sensitivity of mean pressure to simulation parameters using three different sensitivity indices: RBD-FAST \cite{tarantola2006random} (top row), PAWN \cite{pianosi2015simple} (middle row), and Delta moment-independent analysis \cite{plischke2013global} (bottom row). In each row, we compare one realization of the full high-fidelity data matrix consisting of 237 high-fidelity simulations with the multi-modal ID approximation computed from 25 high-fidelity simulations. Each parameter $\xi_i$ is defined in Table \ref{tab:uncertainties}. Error bars denote 95\% confidence intervals due to sampling of $\widehat{\bm H}$ in the multi-modal ID algorithm, as well as sampling within the sensitivity index algorithm (if applicable). The dominant parameter $\xi_6$ corresponds to the energy deposited by the laser kernel.}
    \label{fig:sens}

\end{figure}

\paragraph{Results and discussion}

In all subsequent results, we use a rank 25 approximation of the data matrices, with $N_S=5$ independent samples of each low-fidelity run. As shown in Figure \ref{fig:psaap:0}, this corresponds to a relative 2-norm error of approximately 5\% when reconstructing the low-fidelity data matrix.
The multi-modal ID algorithm cannot be evaluated using the same $W_1$ distance metric as the toy problems because we can no longer sample the true high-fidelity data matrix on demand. Instead, our final assessment of the proposed method is based on the following:
\begin{enumerate}
\item A visual comparison of a single prediction sample of the low-fidelity data matrix $\bm L$ and high-fidelity data matrix $\bm H$ for both multi-modal and baseline ID is shown in Figure \ref{fig:psaap_h_hat}. Each column of the data matrix is plotted separately. This illustrates qualitatively the failure of ID and that physically reasonable predictions are obtained with multi-modal ID.
\item We choose 5 random columns from the data matrices, and then compare the predictions from multi-modal ID and baseline ID in Figure \ref{fig:psaap:errobar}. For a given column (corresponding to a given model parameter value), 5 predictions of the high-fidelity pressure trace are sampled from each method, based on five samples of the low-fidelity data matrix. The black dotted line shows one sample from the high-fidelity model. Again, this demonstrates that the proposed algorithm regains physically plausible predictions, given stochastic and multi-modal data. However, note that in column 126 the predicted high-fidelity does not ignite. This reflects small differences in the ignition probability between the fidelities, which the proposed method does not correct.
\item In Figure \ref{fig:sens}, the results of a sensitivity analysis on the full 237 high-fidelity pressure traces are shown, compared to the same sensitivity analysis from the multi-modal ID approximation. The quantity of interest is the mean pressure, which is a proxy for whether the combustor ignites successfully or not. We employ three different sensitivity indices to indicate the robustness of this result. The parameter sensitivity correlation ranges from 0.70 (PAWN) to 0.90 (Delta), indicating that the bi-fidelity method yields very similar sensitivities. In all cases, the multi-modal ID and the full high-fidelity data both indicate that $\xi_6$, the energy deposited by the laser, is the dominant parameter in the ignition outcome of these trials.

\end{enumerate}

Overall, the baseline ID algorithm fails to provide reasonable pressure trace predictions on this problem. The pressure traces exhibit stochasticity and are strongly bimodal, corresponding to either an ignition or a non-ignition state. The multi-modal ID algorithm produces the expected bimodal predictions and obtains parameter sensitivities comparable to those of the full high-fidelity data. 

Some limitations remain, in particular, the assumption of a similar interpolation matrix between fidelities. In the present application, a small number of simulations ignite after a long delay ($>300\mu$s), and the randomness inherent in this process means there is limited correlation between fidelities. In addition, the proposed method does not yield a predictive distribution for the basis elements, because the predictions here revert to the single realized high-fidelity model output. Finally, the proposed method yields ignition probabilities based solely on the low-fidelity model samples. In applications for which the cluster probabilities are significantly different, the proposed method would not be appropriate.

\paragraph{Computational cost savings} Using a rank 25 approximation to 237 pressure traces, along with a low-fidelity model at half resolution (and therefore 16$\times$ faster in 3D), the bi-fidelity approximation has 0.16 the cost of running the entire high-fidelity model. The multi-modal ID algorithm itself has negligible cost and can be run on a laptop.

\section{Conclusions}

We have introduced the multi-modal ID method as a bi-fidelity approximation approach for stochastic and multi-modal data. This is an extension of the bi-fidelity ID method, which is based on sampling the low-fidelity model to guarantee agreement in cluster membership of the basis elements. We have demonstrated that the proposed method fixes the basis mismatch issue through numerical examples. The method is demonstrated on a sensitivity analysis for a laser-ignited rocket combustor system, where it can approximate a dataset of high-fidelity simulations for 16\% of the cost, while maintaining correlation (0.70--0.90) with the parameter sensitivities computed on the high-fidelity dataset.

In future work, the assumption of linearity in the reconstruction of each data matrix column from the basis columns could be relaxed using nonlinear methods.

\section*{Acknowledgments}

The authors acknowledge financial support from the US Department of Energy’s National Nuclear Security Administration via the Stanford PSAAP-III Center for the prediction of laser ignition of a rocket combustor (DE-NA0003968).

The authors also extend their gratitude to Dr.\ Gianluca Geracci, Prof.\ Gianluca Iaccarino, Dr.\ Marta D'Elia, and Dr.\ Juan Cardenas for helpful discussion over the course of this work.


\bibliographystyle{unsrtnat}
\bibliography{references}

newpage
\appendix

\section*{Notation}
\begin{description}
    \item[$d$] Dimension of the parametric input space
    \item[$M_L, M_H$] Dimension of the low- and high-fidelity QOI
    \item[$N$] Number of parametric input samples
    \item[$r$] Rank of the low-rank approximation
    \item[$K$] Number of modes in the mixture distribution
    \item[$N_S$] Number of low-fidelity samples per input for the stochastic case
    \item[$S$] Number of high-fidelity samples generated by the stochastic ID method

    \item[$\boldsymbol \xi \in \mathbb{R}^d$] Vector of parametric inputs
    \item[$\bm \omega$] Hidden random variables representing stochasticity
    \item[$\boldsymbol v_L(\boldsymbol\xi)$] Low-fidelity quantity of interest (QOI) vector
    \item[$\boldsymbol v_H(\boldsymbol\xi)$] High-fidelity quantity of interest (QOI) vector
    \item[$\chi$] Mode or cluster membership
    
    \item[$\mathcal{J}$] Index set of basis columns

    \item[$\bm{L} \in \mathbb{R}^{M_L \times N}$] Low-fidelity data matrix
    \item[$\bm{H} \in \mathbb{R}^{M_H \times N}$] High-fidelity data matrix
    \item[$\bm{L}_r, \bm{H}_r$] Low- and high-fidelity basis (skeleton) matrices
    \item[$\bm C_L \in \mathbb{R}^{r \times N}$] Interpolation coefficient matrix
    \item[$\widehat{\bm{L}} \in \mathbb{R}^{M_L \times N}$] Rank-$r$ ID approximation of $\bm{L}$
    \item[$\widehat{\bm{H}} \in \mathbb{R}^{M_H \times N}$] Bi-fidelity approximation of $\bm{H}$
\end{description}

\section{Deterministic bi-fidelity interpolative decomposition algorithm}
\label{app:deterministicid}

\begin{algorithm}[h]
\caption{Bi-fidelity interpolative decomposition for deterministic data.}
\label{alg:deterministicid}
\hrulefill

\KwIn{
    $\bm{L}$: LF data matrix \\
    \hspace{1cm} $\{\boldsymbol \xi_{j} \}_{j=1}^{N}$: Parametric inputs used in each column of $\bm{L}$ \\
    \hspace{1cm} $r$: Desired rank of the approximation \\
}

\KwOut{
    $\widehat{\bm H}$: Approximated HF data matrix \\
}

\hrulefill

\BlankLine
Step 1: Compute rank-$r$ interpolative decomposition: $\bm{L} = \bm{L}_r \bm{C}_L$

\BlankLine
Step 2: Extract basis column set $\mathcal{J}$, such that $\bm{L}_r = \bm{L}(:, \mathcal{J})$ \\

\BlankLine
Step 3: $\bm H_r \gets \left[ \boldsymbol v_H(\boldsymbol \xi_{\mathcal{J}[1]}) \ \cdots \ \boldsymbol v_H(\boldsymbol \xi_{\mathcal{J}[r]})  \right] $ \\

\BlankLine
Step 4: $\widehat{\bm H} \gets \bm H_r \bm{C}_L$ \\

\BlankLine
\textbf{Return} $\widehat{\bm H}$

\hrulefill
\end{algorithm}

\section{Bi-fidelity approximation error for linear, low-fidelity operator}
\label{app:linearop}

In this section, we remark on the accuracy of the standard deterministic ID algorithm when $\bm{L} = \bm{T}\bm{H}$, so the low-fidelity is obtained from applying a linear operator to the high-fidelity. One practical example of this is when low-fidelity data is obtained by convolving a Gaussian with the high-fidelity data. For this (non-stochastic) regime, we now derive some results on the high-fidelity approximation error.

First, assuming that $\mathrm{rank}(\bm{H})=r$, an exact rank-$r$ interpolative decomposition of $\bm{H}$ exists for some interpolation rule $\bm{C}_H$, i.e., $\bm{H} = \bm{H}_r \bm{C}_H $. Then, we show the following condition on $\bm{T}$ is required for the bi-fidelity approximation of $\bm{H}$ using the interpolation rule calculated using $\bm{L}$ to be exact.

\begin{proposition}
    Let $\bm{L} = \bm{T} \bm{H}$. If  $\mathrm{rank}(\bm{H})=r$, then the low and high-fidelity interpolation rules $\bm{C}_L = \bm{C}_H$ if and only if $\bm{L}_r$ has full column rank.
\end{proposition}

\begin{proof}
    There exists $\bm{C}_H$ such that $\bm{H} = \bm{H}_r \bm{C}_H $. Furthermore, if $\bm{L}_r$ has full column rank then there exists $\bm{C}_L$ such that $ \bm{L} = \bm{L}_r \bm{C}_L $.
    Both fidelities share the same basis set, so $\bm{L}_r=\bm{T}\bm{H}_r$. Then:
    \begin{align*}
    \bm{L} &= \bm{T} \bm{H}, \\
        \bm{L}_r \bm{C}_L &= \bm{T} \bm{H}_r \bm{C}_H, \\
    \bm{L}_r \bm{C}_L &= \bm{L}_r \bm{C}_H.
    \end{align*}
    If $\bm{L}_r$ has full column rank, we can apply the left pseudo-inverse $\bm{L}_r^+ = (\bm{L}_r^T \bm{L}_r)^{-1} \bm{L}_r^T$:
    \begin{equation}
        \bm{C}_L = \bm{C}_H.
    \end{equation}
    Conversely, if $\bm{L}_r$ is rank-deficient, then there exists a nonzero null space component $\bm{v}$ such that:
    \begin{equation}
        \bm{L}_r \bm{v} = \bm{0}.
    \end{equation}
    In this case, the equation $\bm{L}_r \bm{C}_L = \bm{L}_r \bm{C}_H$ does not uniquely determine $\bm{C}_L$, as any solution satisfies:
    \begin{equation}
        \bm{C}_L = \bm{C}_H + \bm{v}.
    \end{equation}
    Thus, $\bm{C}_L = \bm{C}_H$ if and only if $\bm{L}_r$ has full column rank.
\end{proof}

A necessary (but not sufficient) condition for full column rank of $\bm{L}_r$, so that the exact relation above holds, is $\mathrm{rank}(\bm{T}) \ge \mathrm{rank}(\bm{H})$.

In the approximate case, $\mathrm{rank}(\bm{H}) > r$, and as before $\bm{L} = \bm{T}\bm{H}$. In this case, only an error bound can be provided. Similar to the analysis in \cite{hampton2018}, it is straightforward to show that the high-fidelity approximation error is bounded from below as

\begin{equation}
    \frac{ \| \widehat{\bm{L}} - \bm{L} \|_F}{ \|\bm{T} \|_2 } \le \| \widehat{\bm{H}} - \bm{H} \|_F .
\end{equation}
If $\bm{T}$ is invertible, then the bi-fidelity error is bound from above via
\begin{equation}
    \| \widehat{\bm{H}} - \bm{H} \|_F \le \| \bm{T}^{-1} \|_2 \| \widehat{\bm{L}} - \bm{L} \|_F.
\end{equation}
We can refine this bound further using the result from \cite{martinsson2011randomized} that $\| \widehat{\bm{L}} - \bm{L} \|_F \le \sqrt{r(N-r)+1} \sigma_{r+1}(\bm{L})$, where $r$ is the rank of the interpolative decomposition, $N$ is the number of columns of $\bm{L}$, and $\sigma_{r+1}(\bm{L})$ is the $(r+1)$-st greatest singular value of $\bm{L}$:
\begin{equation}
    \| \widehat{\bm{H}} - \bm{H} \|_F \le \| \bm{T}^{-1} \|_2 \sqrt{r(N-r)+1} \sigma_{r+1}(\bm{L}).
\end{equation}

\newpage

\section{Stochastic bi-fidelity interpolative decomposition algorithm}
\label{app:stochid}

\begin{algorithm}
\caption{Sample $S$ approximations of the HF data matrix using stochastic ID algorithm.}
\label{alg:1}

\hrulefill

\KwIn{
    $\{ \bm{L}^{(i)} \}_{i=1}^{N_S}$: Repeated samples of low-fidelity data matrix \\
    \hspace{1cm} $\{\boldsymbol \xi_{j} \}_{j=1}^{N}$: Parametric inputs used in each column of the $\bm{L}^{(i)}$ \\
    \hspace{1cm} $r$: Desired rank of the approximation
}
\KwOut{
    $\{ \widehat{\bm H}^{(i)} \}_{i=1}^{S}$: Repeated samples of the approximated high-fidelity data matrix \\
}

\hrulefill

\BlankLine

Step 1: Select basis column set $\mathcal{J}$

\BlankLine
Step 2: $\bm H_r \gets \left[ \boldsymbol v_H(\boldsymbol \xi_{\mathcal{J}[1]}) \ \cdots \ \boldsymbol v_H(\boldsymbol \xi_{\mathcal{J}[r]})  \right] $ \\

\BlankLine
Step 3: Classify cluster memberships $\chi$: \ $\chi^{(j)}_{H} \gets \argmax_\chi P(\chi | \boldsymbol v_H (\boldsymbol \xi_{j})) $ for $j \in \mathcal{J}$ \\

\BlankLine
Step 4: \For{$i \gets 1$ \textbf{to} $S$}{
    $\bm{L} \gets $ \texttt{resample}($\mathcal{J}, \bm{L}^{(1)}, \dots, \bm{L}^{(N_S)}, \chi^{(1)}_{H}, \dots, \chi^{(r)}_{H})$ \\
    $\bm{C}_L \gets \argmin_{ \bm C_L} \| \bm L - \bm L(:, \mathcal{J}) \bm C_L \|_F^2 + \lambda \| \bm C_L \|_F^2 $\\
    $\widehat{\bm{H}}^{(i)} \gets \bm{H}_r \bm{C}_L$ 
}

\textbf{Return} $\{ \widehat{\bm{H}}^{(i)} \}_{i=1}^{S}$

\hrulefill
\BlankLine

\end{algorithm}

\begin{algorithm}
\caption{\texttt{resample()} used in Algorithm \ref{alg:1} to sample a low-fidelity data matrix.}
\label{alg:2}

\hrulefill

\KwIn{
    $\mathcal{J}$: Index set of basis columns \\
    \hspace{1cm} $\{ \bm{L}^{(i)} \}_{i=1}^{N_S}$: Repeated samples of low-fidelity data matrix \\
    \hspace{1cm} $\{ \chi^{(i,j)}_{L} \}_{i=1,j=1}^{N, N_S}$: Cluster labels of low-fidelity data columns \\
    \hspace{1cm} $\{ \chi^{(j)}_{H} \}_{j \in \mathcal{J}}$: Cluster labels of high-fidelity basis columns \\
}

\KwOut{
    $\bm{L}$: Resampled LF data, with basis columns matching HF cluster labels
}

\hrulefill

\BlankLine

Step 1: Initialize the resampled LF data matrix $\bm{L} \in \mathbb{R}^{M_L \times N}$ \\

\BlankLine
Step 2: 
\For{$j \in \mathcal{J}$}{
    \BlankLine
    $\mathcal{I}_{\text{match}} \gets \{i \ | \ \chi^{(i,j)}_{L} = \chi^{(j)}_{H} \ \mathrm{for} \ i=1, 2, ..., N_S \}$ \\

    $k \gets \texttt{random.choice}(\mathcal{I}_{\text{match}})$ \\

    $\bm{L}(:, j) \gets \bm{L}^{(k)}(:, j)$ \\
}

\BlankLine
Step 3: 
\For{$j \notin \mathcal{J}$}{
    $k \gets \texttt{random.choice}(1, \dots, N_S)$ \\
    $\bm{L}(:, j) \gets \bm{L}^{(k)}(:, j)$ \\
}

\textbf{Return} $\bm{L}$

\hrulefill
\BlankLine

\end{algorithm}

\newpage

\section{Uncertainty distribution in laser-ignited combustor simulations}

\begin{table}[ht]
    \centering
    \begin{tabular*}{\linewidth}{@{\extracolsep{\fill}} l c c}
        \toprule
        \textbf{Uncertainty ID} & \textbf{Description} &  \textbf{Distribution}  \\
        \midrule
        $\xi_1\!: \Delta x_r$          &  Radial focal location  &  $N \sim (0.29 \, \text{mm}, 0.04 \, \text{mm}^2)$ \\
        $\xi_2\!: \Delta x_l$          &  Streamwise focal imprecision  &  $N \sim (-0.54 \, \text{mm}, 0.20 \, \text{mm}^2)$ \\
        $\xi_3\!: l_{axial}$           & Laser axial length   &  $U \sim [1.44 \text{mm}, 2.16 \text{mm}]$ \\
        $\xi_4\!: \alpha = l_{axial}/2R1$      & Aspect ratio   &  $U \sim [2.0, 2.5]$ \\
        $\xi_5\!: \beta = R1/R2$       & Lobe radii ratio   &  $U \sim [1.1, 2.1]$ \\
        $\xi_6\!: E$                   & Energy deposited   & $U \sim [20 \, \text{mJ}, 54 \, \text{mJ}]$ \\
        $\xi_7\!: TF_{\beta}$          & Thickened flame model parameter & $U \sim [0.5,0.60]^*$ \\
        $\xi_8\!: TF_{SL,0}$           & Thickened flame model laminar flame speed & $U \sim [0.0098,0.011]$  \\ 
        $\xi_{9}\!: \dot{\omega}$     & Reaction rate & $U \sim [3.4 \times 10^9,3.8 \times10^9]$ \\
        $\xi_{10}\!: C{s}$             & Smagorinsky constant & $U \sim [0.15,0.17]$ \\
        $\xi_{11}\!: \dot{m}_{ox}$     & Mass flow rate oxygen &  $U \sim [6.12 \text{g/s},6.83 \text{g/s}]$ \\
        $\xi_{12}\!: \dot{m}_{fuel}$   & Mass flow rate fuel   & $U \sim [1.97 \text{g/s}, 2.17 \text{g/s}]$ \\
        \bottomrule
    \end{tabular*}
    \caption{Summary of the uncertainty parameters $\bm\xi$ used in the uncertainty quantification of the laser-ignited rocket combustor test case. Probability distributions are denoted as $U \sim [a, b]$ (uniform probability between $a$ and $b$) or $N \sim (\mu, \sigma^2)$ (normal distribution with mean $\mu$ and variance $\sigma^2$).}
    \label{tab:uncertainties}
\end{table}

\end{document}